\documentclass{commstat}

\usepackage{graphicx}
\usepackage{amssymb}
\usepackage{amsthm}
\usepackage{subfigure}

\theoremstyle{plain}
\newtheorem{theorem}{Theorem}[section]

\newtheorem{lemma}[theorem]{Lemma}

\theoremstyle{remark}
\newtheorem{remark}[theorem]{Remark}

\newtheorem{conjecture}[theorem]{Conjecture}

\begin{document}

\articletype{This manuscript is to appear in \\ {\em Communications in Statistics---Theory and Methods}, \\   after a revision suggested by the referees under the new title:\\ 
{\small {\sf Characterizations of Arcsin and Related Distributions Based on a New Generalized Unimodality}} \\
with the DOI: {\tt 10.1080/03610926.2015.1006788}}

\title{{\itshape  Another Generalization of Unimodality}}

\author{{\sc Hazhir Homei}{$^{\ast}$\thanks{$^\ast$Corresponding author. Email: homei@tabrizu.ac.ir
\vspace{6pt}} \\
{\em{Department of Statistics, Faculty of Mathematical
Sciences, \\ University of Tabriz, 29 Bahman Boulevard, P.O.Box 51666--17766, IRAN.}}}\\\received{\today} }

\maketitle

\begin{abstract}
A general characterization for $\alpha$-unimodal distributions was provided by Alamatsaz (1985) who  later introduced a multivariate extension of them (Alamatsaz 1993). Here, by solving the related equations, another generalization for unimodality is presented. As a result of this generalization, a simpler  proof of a conjecture, as well as  a characterization for  generalized arcsin distributions and some generalizations of the author's earlier works, have been obtained. Last, but not the least, it is shown that some elementary methods can be more powerful than some more advanced techniques.

\begin{keywords}  Unimodality; Characterization; Generalized Arcsin Distribution.
\end{keywords}

\begin{classcode} 62E10; 62E15.
\end{classcode}

\end{abstract}

\section{Introduction}
Unimodality property is a long standing problem in  distribution theory, and has been studied by several authors; see \cite{alamat85} and its references.  The random variable $Z$ is called unimodal when there are independent random variables $X$ and $U$ such that   $$Z\,{\buildrel d \over =}\,U\cdot X,\eqno{(1)}$$
where $U$ has a uniform distribution on $[0,1]$. Several generalizations of $(1)$ has been studied in the literature (see e.g. \cite{alamat93}); here we generalized it as follows $$S_n\,{\buildrel d \over =}\,\langle{\bf R},{\bf X}\rangle,\eqno{(2)}$$ where ${\bf R}=(R_1,\cdots,R_n)$ is the Dirichlet random vector (and so $\sum_{i=1}^nR_i=1$) and ${\bf X}=(X_1,\cdots,X_n)$ is an arbitrary random vector which is independent from ${\bf R}$; the symbol $\langle\cdot,\cdot\rangle$ denotes the inner product. Our main motivation for this generalization was solving the forthcoming conjecture.

\subsection{A Conjecture On Arcsin Distributions}
Consider the inner product of two independent random vectors  ${\bf R}=(R_1,\cdots,R_n)$ and ${\bf X}=(X_1,\cdots,X_n)$  defined by
$$S_n\,{\buildrel d \over =}\,\langle{\bf R},{\bf X}\rangle=\sum_{i=1}^{n}R_i\cdot X_i  \qquad (n\geq 2).$$ Now, the  components of the random
vector ${\bf R}$  can be defined  as $R_{i}=U_{(i)}-U_{(i-1)}$  (for $i=1,\cdots,n-1$ and
$R_n= 1-\sum_{i=1}^{n-1} R_i$) where  $U_{(1)},\cdots,U_{(n-1)}$ are
order statistics of a random sample $U_1,\cdots,U_n$ from a uniform
distribution on $[0,1]$ with $U_{(0)}=0$ and $U_{(n)}=1$.
Note that
the distribution of ${\bf R}=(R_1,\cdots,R_n)$ is the Dirichlet
distribution $D_{n}(1,\cdots,1)$;  see (Wang et. al. 2011).

\begin{conjecture}\label{conj}{\rm
If the random
variables $X_1,\cdots,X_n$ are independent and have common Arcsin
distribution on $(-a,a)$, then $S_n$ will have a power semicircle
distribution on $(-a,a)$ with $\lambda=\frac{(n-1)}{2}$, i.e.,
$f(x;\lambda,a)=\frac{1}{\sqrt \pi {a^{2\lambda}}}\frac{\Gamma(\lambda+1)}
{\Gamma(\lambda+\frac{1}{2})}(a^2-x^2)^{\lambda-\frac{1}{2}} \quad (|x|<a)$˜ (see the conclusions of \cite{homei2014}).} \hfill $\blacklozenge$
\end{conjecture}
This conjecture had been proved first for the cases of $n=2,3,4$ and later for all $n$'s; see \cite{homei2013}.  In this short note we prove the above conjecture, and generalize it to generalized arcsin distributions,  by employing  simple methods of analysis based on first
principles which is appropriate for classroom use in advanced
undergraduate or elementary graduate courses in probability and
statistics.

\section{The Main Result}
In order to prove our main result, we need the following lemma.

\begin{lemma}\label{lemm}
For all positive integers $r \in \mathbb N$, we have
$$\sum_{i_1+\cdots+i_n=r}^{}{r\choose
i_{1}, i_{2}, \cdots, i_{n}}
\frac{\Gamma(\frac{1}{2}+i_1)}{\Gamma(\frac{1}{2})}\cdots\frac{\Gamma(\frac{1}{2}+i_n)}
{\Gamma(\frac{1}{2})}=\frac{\Gamma(\frac{n}{2}+r)}
{\Gamma(\frac{n}{2})}.$$
\end{lemma}

\begin{proof}
Let the distribution of $f(x|\textbf{p})$ be multinomial with the
parameters ${\bf p}=(p_1,\cdots,p_n)$, and assume that ${\bf
p}=(p_1,\cdots,p_n)$ has Dirichlet distribution
$D_{n}(\frac{1}{2},\cdots,\frac{1}{2})$. So, the distribution of
$f(x)$ can be calculated, and the lemma is proved considering the
fact that the sum of $f(x)$ on its support equals to one. The function $f(x)$ is
called Dirichlet-Multinomial distribution (see chapters 6 and 7 of
\cite{K}).
\end{proof}


\begin{theorem}\label{theorem}
Assume that  the random variables $X_1,\cdots,X_n$ are independent
and have common Arcsin distribution on (-a,a). Then $S_n$ will have
a power semicircle distribution on (-a,a) with
$\lambda=\frac{n-1}{2}$.
\end{theorem}
\begin{proof}
Since, for any $\sigma$ and $\xi$, we have
$$\sigma S_n+\xi = \sum_{i=1}^{n}R_{i}(\sigma X_i+\xi), \eqno(3)
$$
then without loss of generality we can assume that $a=1$. We find the $r^{\rm
th}$ moment of $S_n$ as follows:
$$E({S_n}^r)=\sum_{i_1+\cdots+i_n=r}^{}\frac{r!}{{i_1}!\cdots {i_n}!}E({R_1}^{i_1}\cdots {R_n}^{i_n})E({X_1}^{i_1})\cdots E({X_n}^{i_n}).$$
By using the Dirichlet distribution, we have
$$E({S_n}^r)=\sum_{i_1+\cdots+i_n=r}^{}\frac{r!}{{i_1}!\cdots {i_n}!}{(n-1)!\frac{\Gamma(i_1+1)\cdots
\Gamma(i_n+1)}{\Gamma(r+n)}}E({X_1}^{i_1})\cdots
E({X_n}^{i_n}).\eqno(4)$$ One can show that
$$E({X_j}^{i_j})=\frac{1}{2}\frac{(1+(-1)^{i_j})
\Gamma(\frac{1}{2}+\frac{i_j}{2})}{\sqrt{\pi}\Gamma(1+\frac{i_j}{2})},
\textrm{ for } j=1,\cdots,n,$$(see page 153 of  \cite{bN}). So,
$E({S_n}^r)=$
$$\sum_{i_1+\cdots+i_n=r}^{}\frac{r!}{{i_1}!\cdots {i_n}!}{(n-1)!
\frac{\Gamma(i_1+1)\cdots \Gamma(i_n+1)}{\Gamma(r+n)}}$$
$${\frac{1}{2}\frac{(1+(-1)^{i_1})
\Gamma(\frac{1}{2}+\frac{i_1}{2})}{\sqrt{\pi}\Gamma(1+\frac{i_1}{2})}}\cdots
\frac{1}{2}\frac{(1+(-1)^{i_n})
\Gamma(\frac{1}{2}+\frac{i_n}{2})}{\sqrt{\pi}\Gamma(1+\frac{i_n}{2})}.$$
Since Arcsin distribution is symmetric about zero, the $r^{\rm th}$
moment is zero for odd r. Now we note that for even $r(=2k)$ if
$i_1+\cdots+i_n =r = 2k$ and if ${i_{j}}$ is odd then
$1+(-1)^{i_j}=0$, and so the corresponding summand will equal to
zero. Hence, we assume all ${i_{j}}'{s}$ to be even and so we write
$2i_j$ in place of $i_j$. Thus,
$$E({S_n}^{2k})=\sum_{2i_1+\cdots+2i_n=2k}^{}\frac{(2k)!}{(2{i_1})!\cdots (2{i_n})!}(n-1)!$$
$$\frac{\Gamma(2{i_1}+1)\cdots\Gamma(2{i_n}+1)}{\Gamma(2k+n)}
\frac{1}{2}\frac{2\Gamma(\frac{1}{2}+i_1)}{\Gamma(\frac{1}{2})\Gamma(1+i_1)}\cdots\frac{1}{2}
\frac{2\Gamma(\frac{1}{2}+i_n)}{\Gamma(\frac{1}{2})\Gamma(1+i_n)}$$
$$=\frac{(2k)!(n-1)!}{\Gamma(2k+n)}\sum_{i_1+\cdots+i_n=k}^{}\frac{1}{{i_1}!\cdots{i_n}!}
\frac{\Gamma(\frac{1}{2}+i_1)}{\Gamma(\frac{1}{2})}\cdots
\frac{\Gamma(\frac{1}{2}+i_n)}{\Gamma(\frac{1}{2})}$$
$$=\frac{(2k)!(n-1)!}{\Gamma(2k+n)k!}\sum_{i_1+\cdots+i_n=k}^{}\frac{k!}{{i_1}!\cdots{i_n}!}
\frac{\Gamma(\frac{1}{2}+i_1)}
{\Gamma(\frac{1}{2})}\cdots\frac{\Gamma(\frac{1}{2}+i_n)}{\Gamma(\frac{1}{2})}.$$
By using Lemma~\ref{lemm} we find that
$$E({S_n}^{2k})=\frac{(2k)!(n-1)!}{\Gamma(2k+n)k!}\frac{\Gamma(\frac{n}{2}+k)}
{\Gamma(\frac{n}{2})}.$$ Using the properties of gamma function  we finally
obtain
$$E({S_n}^r)=\left\{
\begin{array}{cl}
  0  & \textrm{ if }\ r=2k+1, \\
  \frac{\Gamma({k}+\frac{1}{2})\Gamma(\frac{n}{2}+\frac{1}{2})}
  {\sqrt{\pi}\Gamma({k}+\frac{n}{2}+\frac{1}{2})} & \textrm{ if }\
  r=2k.
\end{array}
\right.
$$
It can be easily shown that this  is the $r^{\rm th}$ moment of the power
semicircle distribution (see  \cite{bN}). Since $S_n$ is a bounded
random
 variable, its distribution is uniquely determined by its moments (Carleman's
 Theorem, see e.g. \cite{chung}). Thus the proof is complete.
\end{proof}


\begin{remark}{\rm
Using the equation (3) and choosing suitable $\sigma$ and $\xi$ we can assume
that the support of all $X_i$'s are $[0,1]$ in which case the
obtained moments are well-known Beta distributions.}\hfill $\blacklozenge$
\end{remark}

\section{The Case of Common Distributions}
We note that in case all $X_i$'s have a common distribution, Theorem~\ref{theorem}  provides a characterization
of Beta distributions, which is not studied before (see \cite{homei2013} and its references).

\begin{remark}\label{rem1}{\rm When $X_i$'s have a common distribution, their moments can be derived from the
moments of $S_{n}$ by using (4) noting that
$${r+n-1\choose r}E(S_n^r)=\sum_{i_1+\cdots+i_n=r}E(X_1^{i_1})\cdots E(X_n^{i_n})$$
implies
$${r+n-1\choose r}E(S_n^r)=E(X_1^r)+\cdots+E(X_n^r)+\sum_{i_1+\cdots+i_n=r,   i_1,\ldots,i_n\not=r}E(X_1^{i_1})\cdots E(X_n^{i_n})$$
 and so
 $${r+n-1\choose r}E(S_n^r)=nE(X_i^r)+\mathcal{F}\Big(E(X_i),E(X_i^2),\cdots,E(X_i^{r-1})\Big)$$
 for a function $\mathcal{F}$, which finally gives us
 $$E({X_i}^r)=\frac{1}{n}{n+r-1\choose
r}E({S_n}^r)-\mathfrak{F}\Big(E(S_n),E(S_n^2),\cdots,E(S_n^{r-1})\Big)$$
for a function $\mathfrak{F}$ with $r=1, 2, \cdots $ successively. Then the
distribution of $X_{i}$ is characterized (since $X_{i}$'s have
bounded support).}\hfill $\blacklozenge$
\end{remark}

Very similarly to Theorem~\ref{theorem}  we can characterize the
generalized Arcsin distributions  as follows. Before that we need a lemma;
for the definitions see e.g.  \cite{bN}.


\begin{lemma}\label{lem2}
For all positive integers $r \in \mathbb N$, we have
$$\sum_{i_1+\cdots+i_n=r}^{}{r\choose
i_{1}, i_{2}, \cdots, i_{n}}
\frac{\Gamma(a_1+i_1)}{\Gamma(a_1)}\cdots\frac{\Gamma(a_n+i_n)}
{\Gamma(a_n)}=\frac{\Gamma(r+\sum_{i=1}^na_i)}
{\Gamma(\sum_{i=1}^na_i)}.$$
\end{lemma}

\begin{proof}
Just like the proof of Lemma~\ref{lemm}  using the  Dirichlet-Multinomial distribution in chapters 6 and 7 of
\cite{K}.
\end{proof}

\begin{theorem}\label{theorem2}
Assume that  the random variables $X_1,\cdots,X_n$ are independent
and have common  distribution on (-a,a). Then $S_n$ will have a Beta$(n\alpha, n(1-\alpha),-a, 2a)$ distribution on (-a,a) if and only
if  $X_1,\cdots,X_n$ have generalized Arcsin$(\alpha)$
distributions.
\end{theorem}
\begin{proof}
 By Lemma~\ref{lem2} the theorem is proved in the lines of the proof of Theorem~\ref{theorem} and Remark~\ref{rem1}.
\end{proof}

\begin{remark}\label{rem0}
Though this Theorem (\ref{theorem2}) only applies under some very specific assumptions, it is a wild  generalization of Conjecture~\ref{conj}.  Also, some results of \cite{van87} and \cite{homei2012,homei2014} are special cases of Theorem~\ref{theorem2}.
\hfill $\blacklozenge$
\end{remark}

\begin{remark}\label{rem2} We
note that having the bounded support in this work plays an essential
role; but in case all $X_i$'s have Cauchy distributions,
by using  the  characteristic
function (i.e., the Fourier transform) and conditional expectation
one can overcome this problem;
see 
 \cite{homei2013} and the examples of its references.
\hfill $\blacklozenge$
\end{remark}

\begin{remark}\label{rem3}
In the moments method finding the desired distribution may need having sufficient information about the solution of the problem. Of course in using  the method of \cite{homei2014} one should know the Stieltjes transform of the distribution in question, but in the moments method one can approach the solution by the guesses resulted from calculating the moments sequentially.
\hfill $\blacklozenge$
\end{remark}

\begin{remark}\label{rem4}
The moments method that we used depends on the distributions which are to be characterized by their moments.  This holds for all the cases studied in the references of \cite{homei2013}, because all their distributions have bounded support. Though, some distributions which do not have bounded support could be characterized by their moments. So, this method can be applied to a large family of distributions, provided that the uniqueness conditions are satisfied.  It is remarkable that still many researches in various fields (from 1990 until now)   use the moments method for calculating the distributions; see for example \cite{barrera,jk,jk2,stoyanov,stoyanov2} among others.
\hfill $\blacklozenge$
\end{remark}

\section{Conclusions}
The method of this article provides an elementary and direct way for
calculating  the distribution of the inner product of certain
random vectors. So, this goes to say that the Stieltjes transform is
neither the only nor the best way for calculating the distribution of the
inner products of different random vectors, as long as there is no single example which cannot be handled by our proposed method. Also, this
method can be extended for the cases that $X_1,\cdots,X_n$ have
 Beta distributions, which is intended to be studied in a future paper.




\end{document}